\documentclass[a4paper,twoside,10pt]{amsart}
\usepackage[english]{babel}
\usepackage{t1enc}
\usepackage[charter]{mathdesign}
\usepackage{amsthm}

\newcommand{\coloneq}{\mathrel{\vcenter{\baselineskip0.5ex \lineskiplimit0pt
                     \hbox{\scriptsize.}\hbox{\scriptsize.}}}%
                     =}

\newcommand\Ccal{\mathcal C}

\newcommand\Mcal{\mathcal M}
\newcommand\Pcal{\mathcal P}
\newcommand\Nbb{\mathbb N}

\newcommand\Qbb{\mathbb Q}
\newcommand\Rbb{\mathbb R}

\newcommand\loc{\mathrm{loc}}
\newcommand\ACC{\mathrm{ACC}}
\newcommand\AC{\mathrm{AC}}
\newcommand\C{\mathrm{c}}
\newcommand\upto{\nearrow}
\newcommand\NX{{N^1\!X}}
\newcommand\NhX{{\widehat{N}^1\!X}}
\newcommand\NtX{{\widetilde{N}^1\!X}}
\newcommand\eps{\varepsilon}
\newcommand\Mod{\mathop{\mathrm{Mod}}\nolimits}
\renewcommand{\theenumi}{(\alph{enumi})}

\theoremstyle{plain}
\newtheorem{theorem}{Theorem}[section]
\newtheorem{lemma}[theorem]{Lemma}
\newtheorem{proposition}[theorem]{Proposition}
\newtheorem{corollary}[theorem]{Corollary}
\theoremstyle{definition}
\newtheorem{definition}[theorem]{Definition}
\newtheorem{example}[theorem]{Example}
\newtheorem{remark}[theorem]{Remark}
\numberwithin{equation}{section}
\begin{document}
%
%
%
%
\begin{abstract}
In this paper, first-order Sobolev-type spaces on abstract metric measure spaces are defined using the notion of (weak) upper gradients, where the summability of a function and its upper gradient is measured by the ``norm'' of a quasi-Banach function lattice. This approach gives rise to so-called Newtonian spaces. Tools such as moduli of curve families and Sobolev capacity are developed, which allows us to study basic properties of these spaces. The absolute continuity of Newtonian functions along curves and the completeness of Newtonian spaces in this general setting are established.
\end{abstract}
%
%
%
%
\title[Newtonian spaces based on quasi-Banach function lattices]{Newtonian spaces\\based on quasi-Banach function lattices}
\author{Luk\'{a}\v{s} Mal\'{y}}
\date{August 20, 2013}
\address{Department of Mathematics\\Link\"{o}ping University\\SE-581~83~Link\"{o}ping\\Sweden}
\email{lukas.maly@liu.se}
\thanks{The author was partly supported by NordForsk Research Network ``Analysis and Applications'' grant 080151}
\maketitle
%
%
%
%
\section{Introduction}
\label{sec:intro}
The aim of this paper is to build up the basic theory of Newtonian spaces based on quasi-Banach function lattices and eventually show some interesting properties in this general setting. Newtonian spaces are first-order Sobolev-type spaces on abstract metric measure spaces. The interest in first-order analysis in metric spaces was initiated by Haj\l{}asz~\cite{Haj96} in 1996 and the area has been under intensive study ever since. It leads to exciting new results, which can be readily used also when studying functions defined on (not necessarily open) subsets of $\Rbb^n$. We refer the interested reader to, e.g., Ambrosio and Tilli~\cite{AmbTil}, Bj\"{o}rn and Bj\"{o}rn~\cite{BjoBjo}, Haj\l{}asz~\cite{Haj}, Heinonen~\cite{Hei,Hei2}, or Heinonen, Koskela, Shanmugalingam, and Tyson~\cite{HeiKosShaTys}.

If we focus on the classical definition of a Sobolev space $W^{1,p}(\Omega)$ for some open set $\Omega \subset \Rbb^n$, we can see that the Sobolev norm
\[
  \|u\|_{W^{1,p}(\Omega)} = \Big( \|u\|^p_{L^p(\Omega)} + \| \nabla u \|^p_{L^p(\Omega)}\Big)^{1/p}
\]
does not really depend on the vector of the distributional gradient $\nabla u$, but only on its modulus $|\nabla u|$. Owing to the Newton--Leibniz formula, the modulus $|\nabla u|$ can be used to estimate the difference of function values. For illustration, let $\Omega \subset \Rbb^n$, and suppose that $u\in \Ccal^1(\Omega)$ and that $\gamma: [0, l_\gamma] \to \Omega$ is a $\Ccal^1$-curve. Then,
\[
  | u(\gamma(0)) - u(\gamma(l_\gamma)) | = \Bigl| \int_0^{l_\gamma} (u\circ \gamma)'(t)\,dt\Bigr| \le \int_0^{l_\gamma} |\nabla u| |\gamma'(t)|\,dt = \int_\gamma |\nabla u|\,ds,
\]
where $ds$ denotes arc length. The upper gradients (see Section~\ref{sec:prelim}) substitute $|\nabla u|$ in the inequality above and subsequently in the Sobolev norm, giving rise to the Newtonian norm. The upper gradients were introduced in Heinonen and Koskela~\cite{HeiKos0,HeiKos}. Since the upper gradients, unlike the distributional gradients, do not rely on the linear structure of $\Rbb^n$, they can be used to define first-order Sobolev-type spaces on abstract measure spaces.

Shanmugalingam pioneered this approach in~\cite{Sha} to study the Newtonian spaces corresponding to the Sobolev spaces $W^{1,p}$. Bj\"{o}rn and Bj\"{o}rn~\cite{BjoBjo} gave a thorough treatise on these spaces, including their applications in non-linear potential theory. Durand-Cartagena~\cite{Dur} investigated the case $p=\infty$. Tuominen~\cite{Tuo} and A\"{\i}ssaoui~\cite{Ais} generalized the theory so that the underlying function space would be a reflexive Orlicz space. Harjulehto, H\"{a}st\"{o}, and Pere  further developed the theory in~\cite{HarHasPer}, where they discussed the Newtonian spaces based on Orlicz--Musielak variable exponent spaces, where the exponent function was essentially bounded. Mocanu~\cite{Moc2} worked with Banach function spaces as defined in Bennett and Sharpley~\cite[Definition I.1.3]{BenSha}. The paper~\cite{Moc2} however suffers from improper work with equivalence classes, which eventually leads to invalidity of some of the claims therein. Some of the results there also rely on uniform convexity of the function space which considerably lessens the generality.
The latest attempt to discuss the foundations of the Newtonian theory is due to Costea and Miranda~\cite{CosMir} who used the Lorentz $L^{p,q}$ spaces as the underlying function spaces.
For detailed historical notes on the development of the Newtonian theory and its toolbox, we refer the reader to Bj\"{o}rn and Bj\"{o}rn~\cite[Section 1.8]{BjoBjo}.

The present paper develops elements of an omnibus Newtonian theory that encompasses all these results and goes even further. Under very weak assumptions on the measure and the function space, we establish standard tools for the theory. We prove that the natural equivalence classes are in general finer than equality almost everywhere. Further, we will see that Newtonian functions satisfy and can be characterized by a regularity condition in terms of absolute continuity on curves. We also show that the Newtonian space is in fact a (quasi)Banach space. Finer properties of the set of weak upper gradients are then studied in~\cite{Mal2}. Particularly, existence of minimal weak upper gradients is established there.

There are other possible generalizations of Sobolev spaces to metric measure spaces, based on different characterizations of the distributional gradient. For comparison of these approaches, see Haj\l{}asz~\cite{Haj96,Haj}, Bj\"{o}rn and Bj\"{o}rn~\cite[Appendix~B]{BjoBjo}, or Heinonen, Koskela, Shanmugalingam, and Tyson~\cite[Chapter~9]{HeiKosShaTys}.

The paper is structured in the following way. In Section~\ref{sec:prelim}, we define the quasi-Banach function lattices and the Newtonian spaces based on them. We also show that Newtonian functions form a quasi-normed lattice. Section~\ref{sec:cap} is devoted to the Sobolev capacity and its fundamental properties. Then, we introduce the moduli of curve families in Section~\ref{sec:mod}, which leads to the notion of weak upper gradients that is established and studied in Section~\ref{sec:wug}. We prove that Newtonian functions are absolutely continuous on almost every curve and we discuss the equivalence classes in the Newtonian space in Section~\ref{sec:acc}. Finally, we show that the space of Newtonian functions is complete and we prove a Egorov-type theorem in Section~\ref{sec:banach}.
%
%
%
%
\section{Preliminaries}
\label{sec:prelim}
We assume throughout the paper that $\Pcal = (\Pcal, d, \mu)$ is a metric measure space equipped with a metric $d$ and a $\sigma$-finite Borel regular measure $\mu$. In our context, Borel regularity means that all Borel sets in $\Pcal$ are $\mu$-measurable and for each $\mu$-measurable set $A$ there is a Borel set $D\supset A$ such that $\mu(D) = \mu(A)$. The connection between $d$ and $\mu$ is given by the condition that every ball in $\Pcal$ has finite positive measure. Let $\Mcal(\Pcal, \mu)$ denote the set of all extended real-valued $\mu$-measurable functions on $\Pcal$. The set of extended real numbers, i.e., $\Rbb \cup \{\pm \infty\}$, will be denoted by $\overline{\Rbb}$. The symbol $\Nbb$ will denote the set of positive integers, i.e., $\{1,2, \ldots\}$. The open ball centered at $x\in \Pcal$ with radius $r>0$ will be denoted by $B(x,r)$.

A linear space $X = X(\Pcal, \mu)$ of equivalence classes of functions in $\Mcal(\Pcal, \mu)$ is said to be a \emph{quasi-Banach function lattice} (further abbreviated as qBFL) over $(\Pcal, \mu)$ equipped with the quasi-norm $\|\cdot\|_X$ if the following axioms hold:
\begin{enumerate}
  \renewcommand{\theenumi}{(P\arabic{enumi})}
	\setcounter{enumi}{-1}
  \item \label{df:qBFL.initial} $\|\cdot\|_X$ determines the set $X$, i.e., $X = \{u\in \Mcal(\Pcal, \mu)\colon \|u\|_X < \infty\}$;
  \item \label{df:qBFL.quasinorm} $\|\cdot\|_X$ is a \emph{quasi-norm}, i.e., 
  \begin{itemize}
    \item $\|u\|_X = 0$ if and only if $u=0$ a.e.,
    \item $\|au\|_X = |a|\,\|u\|_X$ for every $a\in\Rbb$ and $u\in\Mcal(\Pcal, \mu)$,
    \item there is a constant $c\ge 1$, the so-called \emph{modulus of concavity}, such that the inequality $\|u+v\|_X \le c(\|u\|_X+\|v\|_X)$ holds for all $u,v \in \Mcal(\Pcal, \mu)$;
  \end{itemize}
  \item $\|\cdot\|_X$ satisfies the \emph{lattice property}, i.e., if $|u|\le|v|$ a.e., then $\|u\|_X\le\|v\|_X$;
    \label{df:BFL.latticeprop}
  \renewcommand{\theenumi}{(RF)}
  \item \label{df:qBFL.RF} $\|\cdot\|_X$ satisfies the \emph{Riesz--Fischer property}, i.e., if $u_n\ge 0$ a.e.\@ for all $n\in\Nbb$, then $\bigl\|\sum_{n=1}^\infty u_n \bigr\|_X \le \sum_{n=1}^\infty c^n \|u_n\|_X$, where $c\ge 1$ is the modulus of concavity. Note that the function $\sum_{n=1}^\infty u_n$ needs be understood as a pointwise (a.e.\@) sum.
\end{enumerate}
Note that $X$ contains only functions finite a.e., which follows from~\ref{df:qBFL.quasinorm} and~\ref{df:BFL.latticeprop}. In other words, if $\|u\|_X<\infty$, then $|u|<\infty$ a.e.
A quasi-Banach function lattice is normed, and thus called a \emph{Banach function lattice} (BFL) if the modulus of concavity is equal to~$1$. 

In the further text, we will slightly deviate from this rather usual definition of quasi-Banach function lattices. Namely, we will consider $X$ to be a linear space of functions defined everywhere instead of equivalence classes defined a.e. Then, the functional $\|\cdot\|_X$ is really only a quasi-seminorm.

Throughout the paper, we will also assume that the quasi-norm $\|\cdot\|_X$ is \emph{continuous}, i.e., if $\|u_n - u\|_X \to 0$ as $n\to\infty$, then $\|u_n\|_X \to \|u\|_X$. The continuity of $\|\cdot\|_X$ in normed spaces follows from the triangle inequality. On the other hand, if the space $X$ is merely quasi-normed, then there is an equivalent continuous quasi-norm due to the Aoki--Rolewicz theorem, see Proposition H.2 in Benyamini and Lindenstrauss~\cite{BenLin}. Its proof shows that such an equivalent quasi-norm retains the lattice property.

It is worth noting that the Riesz--Fischer property is actually equivalent to the completeness of the quasi-normed space $X$, given that the conditions \ref{df:qBFL.initial}--\ref{df:BFL.latticeprop} are satisfied and the quasi-norm is continuous, see Zaanen~\cite[Lemma 101.1]{Zaa}, where the equivalence for normed function lattices is discussed but the proof works even in the case of quasi-normed function lattices. The equivalence was first observed by Halperin and Luxemburg~\cite{HalLux} who defined the Riesz--Fischer property in a slightly different way.

Let us now take a look at some examples of function spaces to appreciate the generality of such a setting.
%
%
\begin{example}
(a) All \emph{(quasi)Banach function spaces}, further abbreviated as (q)BFS, are trivially (q)BFL's, as they satisfy not only \ref{df:qBFL.initial}--\ref{df:BFL.latticeprop}, but also the following three axioms:
\begin{enumerate}
  \renewcommand{\theenumi}{(P\arabic{enumi})}
	\setcounter{enumi}{2}
  \item $\|\cdot\|_X$ satisfies the \emph{Fatou property}, i.e., if $0\le u_n \upto u$ a.e., then $\|u_n\|_X\upto\|u\|_X$;
  \item \label{df:BFS.L8loc} if a measurable set $E \subset \Pcal$ has finite measure, then $\|\chi_E\|_X < \infty$;
  \item \label{df:BFL.locL1} for every measurable set $E\subset \Pcal$ of a finite measure there is $C_E>0$ such that $\int_E |u|\,d\mu \le C_E \|u\|_X$ for every measurable function $u$.
\end{enumerate}
Note that the Fatou property implies the Riesz--Fischer property. Condition~\ref{df:BFL.locL1} describes that $X$ is continuously embedded into $L^1_\loc(\Pcal, \mu)$. As particular examples of BFS's we can list $L^p(\Pcal, \mu)$ spaces if $p\in[1, \infty]$, the variable exponent spaces $L^{p(\cdot)}(\Pcal, \mu)$ for $p:\Pcal \to [1, \infty]$, Orlicz spaces, Lorentz and Marcinkiewicz spaces. For a detailed treatise on Banach function spaces, see Bennett and Sharpley~\cite{BenSha}.

(b) $L^p(\Pcal, \mu)$ spaces, where $0<p<1$, are qBFL's, but not qBFS's as they fail the local embedding into $L^1$.

(c) The spaces $L^1(\Pcal, \mu) \cap L^p(\Pcal, \mu)$, where $0<p<1$, are qBFS's. The quasi-norm is given as $\|\cdot\|_{L^1} + \|\cdot\|_{L^p}$. If $\mu(\Pcal) = \infty$, then these spaces are not normable. On the other hand, if $\mu(\Pcal) < \infty$, then the quasi-norm is equivalent to the $L^1$ norm. The functions lying in this space have peaks controlled by the $L^1$ norm, whereas their rate of decay ``at infinity'' is controlled by the $L^p$ norm.

(d) $L^p(\Pcal, \mu)$ spaces, where $p\in(0, \infty]$, with an additional condition that forces the function value to be zero at some point $x_0 \in \Pcal$, e.g.,
\[
  \| u \|_X = \|u\|_{L^p(\Pcal, \mu)} + \sum_{k=1}^\infty \frac{1}{\mu(B(x_0, 2^{-k}))} \int_{B(x_0, 2^{-k})} |u|\,d\mu,
\]
are (q)BFL's, but not (q)BFS's as they fail~\ref{df:BFS.L8loc}, i.e., they do not contain characteristic functions of all measurable sets of finite measure.

(e) The weak $L^1$ space, also denoted by $L^{1,\infty}(\Pcal, \mu)$, is a qBFL, but not a qBFS as it fails the local embedding into $L^1$.

(f) Spaces of continuous, differentiable, or Sobolev functions are not BFL's as they fail to comply with the lattice property.
\end{example}
The readers interested in the abstract theory of partially ordered linear spaces are referred to Luxemburg and Zaanen~\cite{LuxZaa} and Zaanen~\cite{Zaa}, where normed function lattices, among other things, are discussed.

By a \emph{curve} in $\Pcal$ we will mean a rectifiable non-constant continuous mapping from a
compact interval. Thus, a curve can be (and we will always assume that all curves are) parametrized by arc length $ds$, see e.g.\@ 
Heinonen~\cite[Section~7.1]{Hei}. Note that every curve is Lipschitz continuous with respect to its arc length parametrization. The family of all non-constant rectifiable curves in $\Pcal$ will be denoted by $\Gamma(\Pcal)$. By abuse of notation, the image of a curve $\gamma$ will also be denoted by $\gamma$. 

Now, we shall introduce the upper gradients, which are used as a substitute for the modulus of the usual weak gradient in the definition of Newtonian spaces. They were originally introduced by Heinonen and Koskela in~\cite{HeiKos0,HeiKos} under the name very weak gradients.
%
%
\begin{definition}
\label{df:ug}
  Let $u: \Pcal \to \overline{\Rbb}$. Then, a Borel function $g: \Pcal \to [0, \infty]$ is called an \emph{upper gradient} of $u$ if
\begin{equation}
 \label{eq:ug_def}
 |u(\gamma(0)) - u(\gamma(l_\gamma))| \le \int_\gamma g\,ds
\end{equation}
for all curves $\gamma: [0, l_\gamma]\to\Pcal$. To make the notation easier, we are using the convention that $|(\pm\infty)-(\pm\infty)|=\infty$.
\end{definition}
Observe that the upper gradient of a function is by no means given uniquely. Indeed, if we have a function $u$ and its upper gradient $g$, then $g+h$ is another upper gradient of $u$ whenever $h$ is a non-negative Borel function.

The following lemma shows that we can easily find an upper gradient of a linear combination of functions whose upper gradients are known.
%
%
\begin{lemma}
\label{lem:ug.linearity}
Let $g$ and $h$ be upper gradients of $u$ and $v$, respectively, and $a\in\Rbb$. Then, $|a|g$ and $g + h$ are upper gradients of $au$ and $u + v$, respectively.
\end{lemma}
\begin{proof}
This follows immediately from Definition~\ref{df:ug}.
\end{proof}
Now that we have established upper gradients, we can define analogues of Sobolev spaces on metric measure spaces.
%
%
\begin{definition}
Whenever $u\in \Mcal(\Pcal, \mu)$, let
\[
  \|u\|_{\NX} = \| u \|_X + \inf_g \|g\|_X,
\]
where the infimum is taken over all upper gradients $g$ of $u$.
The \emph{Newtonian space} based on $X$ is the space
\[
  \NX = \NX (\Pcal, \mu)= \{u\in\Mcal(\Pcal, \mu): \|u\|_{\NX} <\infty \}.
\]
Let us point out that we assume that functions are defined everywhere, and not just up to equivalence classes $\mu$-almost everywhere. This is essential for the notion of upper gradients since they are defined by a pointwise inequality.

We also define the space of natural equivalence classes given by $\NtX = \NX / \mathord\sim$\,, where the equivalence relation $u\sim v$ is determined by $\|u-v\|_\NX = 0$.
\end{definition}
Note that we follow the notation of Bj\"{o}rn and Bj\"{o}rn~\cite{BjoBjo}, where $\NX$ denotes the space of functions defined everywhere while $\NtX$ denotes the space of equivalence classes. Some authors, e.g., Shanmugalingam~\cite{Sha}, Tuominen~\cite{Tuo} and Mocanu~\cite{Moc2}, use the corresponding symbols the other way around.

We will prove in Corollary~\ref{cor:nat.eq.classes} that the equivalence classes we have just defined are in general finer than the classes of $\mu$-almost everywhere equality.
%
%
\begin{remark}
The theory of upper gradients becomes pathological in some cases and the corresponding Newtonian spaces are rendered trivial in the sense that $\NX = X$. Obviously, if $\Pcal$ does not contain any non-constant rectifiable curves, then the zero function is an upper gradient of any function $u\in X$, and hence $\|u\|_{\NX} = \|u\|_X$. The Koch snowflake provides us with a simple example of such a metric space~$\Pcal$. This exceptional case has already been observed in older papers on Newtonian spaces. However, the following example shows that there are other situations in which $\NX$ becomes degenerate.
\end{remark}
%
%
\begin{example}
\label{exa:trivial_N1p}
Let $X = L^p([0,1])$, where $p\in(0,1)$. Suppose that the set $\{q_i: i\in\Nbb\}$ contains all rational numbers within $[0, 1]$. For $k\in\Nbb$, let
\[
  g_k (x) = \frac{1}{k} \sum_{i=1}^\infty \frac{4^{-i/p}}{|x-q_i|}\,, \quad x\in[0,1].
\]
Then, $\|g_k\|_X = \|g_1\|_X / k < \infty$ for all $k\in\Nbb$. Nevertheless, if we consider an arbitrary curve $\gamma$, then $\int_\gamma g_k = \infty$. Therefore, all $g_k$ are upper gradients of any function $u\in X$. Hence,
\[
  \|u\|_X \le \|u\|_\NX \le \|u\|_X + \frac{\|g_1\|_X}{k} \to \|u\|_X\,\quad\mbox{as }k\to\infty,
\]
which proves that $\NX = X$.
\end{example}
Note that a similar example can be produced even for $X = L^p([0,1]^n)$, where $n> 1$ and $p\in(0,1)$. In that case, let
\[
  g_k (x_1, x_2, \ldots, x_n) = \frac{1}{k} \sum_{i=1}^\infty \sum_{j=1}^n \frac{4^{-i/p}}{|x_j-q_i|}\,, \quad (x_1, x_2, \ldots, x_n) \in[0,1]^n,
\]
for any $k\in\Nbb$, where the set $\{q_i: i\in\Nbb\}$ consists of all rational numbers within the interval $[0, 1]$. Then, all $g_k\in X$ are upper gradients of any function $u\in X$.

In the following two claims, we shall see that $\NX$ is not only a linear space, but also a lattice. Furthermore, the functional $\|\cdot\|_\NX$ is a (quasi)seminorm on $\NX$.
%
%
\begin{proposition}
\label{pro:NX_norm}
The functional $\|\cdot\|_\NX$ is a seminorm on $\NX$ and a norm on $\NtX$, given that $X$ is a Banach function lattice. If $X$ is just a quasi-Banach function lattice, then $\|\cdot\|_\NX$ is a quasi-seminorm on $\NX$ and a quasi-norm on $\NtX$. Moreover, the modulus of concavity remains the same as in $X$.
\end{proposition}
\begin{proof}
Let $\eps>0$, $a\in\Rbb$, and $u,v\in \NX$. Then, there are upper gradients $g,h \in X$ of $u, v$, respectively, such that
\[
  \|u\|_X + \|g\|_X < \|u\|_\NX + \eps, \quad\mbox{and}\quad \|v\|_X + \|h\|_X < \|v\|_\NX + \eps.
\]
Suppose $c\ge 1$ is the modulus of concavity of $X$. Lemma~\ref{lem:ug.linearity} yields that $g+h$ is an upper gradient of $u+v$. Thus,
\begin{align*}
  \|u+v\|_\NX & \le \|u+v\|_X + \|g+h\|_X \\
  & \le c (\|u\|_X + \|v\|_X + \|g\|_X + \|h\|_X) < c(\|u\|_\NX + \|v\|_\NX + 2\eps).
\end{align*}
Letting $\eps \to 0$ proves the triangle inequality. Since $|a|g$ is an upper gradient of $au$, we have
\[
  \|au\|_\NX \le \|au\|_X + \||a|g\|_X = |a|(\|u\|_X + \|g\|_X) \le |a|(\|u\|_\NX + \eps),
\]
which leads to $\|au\|_\NX \le |a|\,\|u\|_\NX$. Similarly, we obtain $\|u\|_\NX \le |a|^{-1} \|au\|_\NX$ for $a\neq 0$. Consequently, $\|au\|_\NX = |a|\,\|u\|_\NX$.
\end{proof}
%
%
\begin{remark}
\label{rem:N1rX}
If the $r$th power of $\|\cdot\|_X$ is subadditive for some $r \in (0, \infty)$, we may also consider the functional
\[
  \|u\|_{N^1_rX} = \Bigl( \|u\|_X^r + \inf_g \|g\|_X^r\Bigr)^{1/r},
\]
where the infimum is taken over all upper gradients $g$ of $u$. It is easy to see that there is a constant $c'\ge 1$ such that $\|u\|_\NX / c' \le \|u\|_{N^1_rX} \le c' \|u\|_\NX$ for all measurable functions~$u$. Similarly as in Proposition~\ref{pro:NX_norm}, we can show that $\| \cdot \|_{N^1_rX}$ is a quasi-seminorm on $\NX$ whose modulus of concavity equals the modulus of concavity of $X$. Furthermore, the $r$th power of  $\| \cdot \|_{N^1_rX}$ is subadditive as well.
\end{remark}
%
%
\begin{theorem}
\label{thm:N1X-lattice}
The space $\NX$ is a\/ \emph{lattice,} i.e., if $u,v \in \NX$, then 
\[
  \max\{u,v\},\min\{u,v\}, |u|, u^+, u^- \in \NX.
\]
\end{theorem}
\begin{proof}
If $g,h\in X$ are upper gradients of $u,v\in\NX$, respectively, then we can easily see that $g+h$ is an upper gradient of $\max\{u,v\}$. All other functions in the theorem can be expressed using $\max$.
\end{proof}
%
%
\begin{remark}
The lattice property~\ref{df:BFL.latticeprop} of a linear function space is a stronger requirement, i.e., if a function space has the lattice property, then it is a lattice. The converse implication does not hold as can be seen, e.g., in the set of continuous functions.
\end{remark}
%
%
%
%
\section{Sobolev capacity}
\label{sec:cap}
%
%
In the theory of quasi-Banach function lattices, it is the sets of measure zero that are negligible and do not carry any information about the functions. If we move to first-order analysis within the context of Newtonian spaces, we will see that we need some quantity providing a finer distinction of small sets.
%
%
\begin{definition}
\label{df:capacity}
The \emph{(Sobolev) $X$-capacity} of a set $E\subset \Pcal$ is defined as
\[
  C_X(E) = \inf\{ \|u\|_{\NX}: u\ge 1 \mbox{ on }E\}.
\]
We say that a property of points in $\Pcal$ holds \emph{$C_X$-quasi-everywhere ($C_X$-q.e.\@)} if the set of exceptional points has $X$-capacity zero. Despite the dependence on $X$, we will often write simply \emph{capacity} and \emph{q.e.} whenever there is no risk of confusion of the underlying function space.
\end{definition}
Sometimes it is convenient to restrict the set of functions over which the infimum is taken to determine the capacity of a set.
%
%
\begin{proposition}
\label{pro:capacity_def}
Let $E\subset \Pcal$. Then,
\[
  C_X(E) = \inf\{ \|v\|_{\NX}: \chi_E \le v \le 1\}.
\]
\end{proposition}
\begin{proof}
Obviously, we have $C_X(E) \le \inf\{ \|v\|_{\NX}: \chi_E \le v \le 1\}$. Thus, if $C_X(E) = \infty$, we are done. Suppose now that $C_X(E)<\infty$ and let $\eps>0$. Then, there is $u \in \NX$ with an upper gradient $g\in X$ such that $u\ge 1$ on $E$ and $\|u\|_X + \|g\|_X < C_X(E)+\eps$. Observe that $g$ is an upper gradient of $\max\{\min\{u,1\},0\}$ as can be seen from the proof of Theorem~\ref{thm:N1X-lattice}. Therefore,
\begin{align*}
  C_X(E) + \eps> \|u\|_X + \|g\|_X & \ge \| \max\{\min\{u,1\},0\} \|_X + \|g\|_X \\
  & \ge \| \max\{\min\{u,1\},0\} \|_{\NX} \ge \inf\{ \|v\|_{\NX}: \chi_E \le v \le 1\}.
\end{align*}
Letting $\eps\to 0$ finishes the proof.
\end{proof}
We also obtain an intermediate result, namely, $C_X(E) = \inf\{\|v\|_\NX: \chi_E \le v\}$.

The following lemma serves as a tool for proving the $\sigma$-quasi-additivity of the Sobolev capacity in Theorem~\ref{thm:CX-properties}.
%
%
\begin{lemma}
\label{lem:sup_uniform_bdd}
Let $u_i$, $i=1,2,\ldots$, be uniformly bounded functions with upper gradients $g_i$. Then, $g=\sup_{i\ge1} g_i$ is an upper gradient of $u = \sup_{i\ge1} u_i$.
\end{lemma}
Note that we cannot remove the assumption on uniform boundedness of the functions $u_i$ as it would render the lemma false. Indeed, consider $u_i \equiv i$ with $g_i \equiv 0$ for all $i\ge 1$. Then, $g\equiv 0$ is not an upper gradient of $u \equiv \infty$.
\begin{proof}
Observe that $u(x) - u(y) = \sup_{i\ge1} (u_i(x) - \sup_{j\ge1} u_j(y)) \le \sup_{i\ge1} (u_i(x) - u_i(y))$ whenever $x,y\in\Pcal$. For every curve $\gamma: [0, l_\gamma] \to \Pcal$, we have
\[
  |u(\gamma(0)) - u(\gamma(l_\gamma))| \le \sup_{i\ge 1} |u_i(\gamma(l_\gamma)) - u_i(\gamma(0))| \le \sup_{i\ge1} \int_\gamma g_i\,ds \le \int_\gamma g\,ds.
\qedhere
\]
\end{proof}
The capacity satisfies the following fundamental properties. Particularly, if $X$ is normed, then $C_X$ is an outer measure on $\Pcal$.
%
%
\begin{theorem}
\label{thm:CX-properties}
Let $E, E_1, E_2, \ldots$ be arbitrary subsets of $\Pcal$. Then
\begin{enumerate}
	\item \label{itm:CX-emptyset}
	      $C_X(\emptyset) = 0$;
	\item \label{itm:CX-vs-measure}
	      $\|\chi_E\|_X \le C_X(E)$; in particular, if $C_X(E) = 0$, then $\mu(E)=0$;
	\item \label{itm:CX-monotonicity}
	      if $E_1\subset E_2$, then $C_X(E_1)\le C_X(E_2)$;
	\item \label{itm:CX-subadditivity}
	      $C_X \bigl( \bigcup_{j=1}^\infty E_j \bigr) \le \sum_{j=1}^\infty c^j C_X(E_j)$, where $c\ge 1$ is the modulus of concavity of $X$.
\end{enumerate}
\end{theorem}
\begin{proof}
The proofs of properties~\ref{itm:CX-emptyset},~\ref{itm:CX-vs-measure}, and~\ref{itm:CX-monotonicity} are trivial. Let us focus on~\ref{itm:CX-subadditivity}. If $C_X(E_j) = \infty$ for some $j\in\Nbb$, then~\ref{itm:CX-subadditivity} holds trivially. Suppose now that $C_X(E_j) < \infty$ for every $j\in\Nbb$. For each $E_j$, $j\in\Nbb$, we can hence find $u_j\in \NX$ with an upper gradient $g_j\in X$ such that $\chi_{E_j} \le u_j \le 1$, and $\|u_j\|_X + \|g_j\|_X < C_X(E_j) + (2c)^{-j} \eps$. Let $u = \sup_{j\ge1} u_j$ and $g = \sup_{j\ge1} g_j$. Then, $\chi_{\bigcup_{j=1}^\infty E_j} \le u \le 1$, while $g$ is an upper gradient of $u$ by Lemma~\ref{lem:sup_uniform_bdd}. Hence,
\begin{align*}
  C_X\biggl( \bigcup_{j=1}^\infty E_j \biggr) & \le \|u\|_{\NX} \le \biggl\| \sup_{j\ge 1} u_j\biggr\|_X + \biggl\| \sup_{j\ge 1} g_j\biggr\|_X \le \biggl\| \sum_{j=1}^\infty u_j\biggr\|_X + \biggl\| \sum_{j=1}^\infty g_j\biggr\|_X \\
  & \le \sum_{j=1}^\infty c^j (\|u_j\|_X + \|g_j\|_X) < \sum_{j=1}^\infty \biggl( c^j C_X(E_j) + \frac{ c^j\eps}{(2c)^j}\biggr) = \eps + \sum_{j=1}^\infty c^j C_X(E_j)\,.
\end{align*}
Letting $\eps\to 0$ completes the proof of~\ref{itm:CX-subadditivity}.
\end{proof}
\begin{remark}
In view of Remark~\ref{rem:N1rX}, we may define another Sobolev $X$-capacity of a set $E\subset \Pcal$ as $C_{X,r}(E) = \inf \{\|u\|^r_{N^1_rX} : u \ge 1\mbox{ on }E\}$, where $r\in(0, \infty)$ is chosen so that the $r$th power of $\|\cdot \|_X$ (and hence of $\|\cdot\|_{N^1_rX}$) is subadditive. Then, $C_{X,r}$ and the $r$th power of $C_X$ are equivalent, i.e., there is $c'\ge 1$ such that $C_{X,r}(E) / c' \le C_X(E)^r \le c' C_{X,r}(E)$ for every $E\subset \Pcal$. Furthermore, it can be proven similarly as Theorem~\ref{thm:CX-properties} that $C_{X,r}$ is an outer measure on $\Pcal$ even if $X$ is merely quasi-normed with modulus of concavity strictly greater than $1$.
\end{remark}
All functions in $X$ are finite a.e. The Newtonian functions, however, satisfy a stronger condition, namely, they are finite q.e., which is shown in the following proposition.
%
%
\begin{proposition}
\label{pro:N1X-finite-qe}
If $u\in \NX$, then $C_X(\{ x\in \Pcal: |u(x)| = \infty\}) = 0$.
\end{proposition}
\begin{proof}
Let $E = \{ x\in \Pcal: |u(x)| = \infty\}$. Then, $|u|/k \ge 1$ on $E$ for all $k>0$. Thus,
\[
  C_X(E) \le \biggl\| \frac{|u|}{k} \biggr\|_{\NX} = \frac{\bigl\|\, |u|\, \bigr\|_{\NX}}{k} \to 0 \quad \mbox{as }k\to \infty.
\qedhere
\]
\end{proof}
%
%
%
%
\section{Modulus of a curve family}
\label{sec:mod}
%
%
In this section we define the $X$-modulus, which allows us to measure curve families in terms of the quasi-norm of the space $X$. The $L^p$-modulus of a system of measures on $\Rbb^n$ was originally defined and studied by Fuglede~\cite{Fug}. Heinonen and Koskela then defined the $L^p$-modulus of a family of curves in a metric measure space in~\cite{HeiKos}. The definition below generalizes their approach; however, where they have the $p$th power of $\|\cdot\|_{L^p}$, we use just $\|\cdot\|_{L^p}$. Despite this little modification, the properties of the modulus remain qualitatively the same and, most importantly, it does not affect which of the curve families have modulus equal to zero.
%
%
\begin{definition}
\label{df:curve_families}
For an arbitrary set $E\subset \Pcal$, we define
\[
  \Gamma_E = \{\gamma \in \Gamma(\Pcal): \gamma^{-1}(E) \neq \emptyset\}
  \quad\mbox{and}\quad
  \Gamma_E^+ = \{\gamma \in \Gamma(\Pcal): \lambda^1(\gamma^{-1}(E))> 0 \},
\]
where $\lambda^1$ denotes the (outer) 1-dimensional Lebesgue measure. 
\end{definition}
\begin{remark}
If the set $\gamma^{-1}(E) \subset \Rbb$ is not $\lambda^1$-measurable, then $\lambda^1(\gamma^{-1}(E))> 0$. Observe that $\Gamma_\Pcal = \Gamma(\Pcal)$.
\end{remark}
%
%
\begin{definition}
\label{df:mod}
Let $\Gamma$ be a family of curves in $\Pcal$. The \emph{$X$-modulus} of $\Gamma$ is defined by
\[
  \Mod_X(\Gamma) := \inf \| \rho\|_X,
\]
where the infimum is taken over all non-negative Borel functions $\rho$ that satisfy $\int_\gamma \rho\,ds\ge 1$ for all $\gamma\in\Gamma$.

A claim is said to hold for \emph{$\Mod_X$-almost every} curve (abbreviated \emph{$\Mod_X$-a.e.} curve) if the family of exceptional curves has zero $X$-modulus.
\end{definition}
%
%
\begin{definition}
A curve $\gamma'$ is a \emph{subcurve} of a curve $\gamma: [0, l_\gamma] \to \Pcal$ if, after re\-pa\-ra\-me\-tri\-za\-tion and perhaps reversion, $\gamma'$ is equal to $\gamma|_{[a,b]}$ for some $0\le a < b \le l_\gamma$.
\end{definition}
The following lemma summarizes the basic properties of the $X$-modulus. Many arguments based on the concept of a modulus depend on the fact that a certain family of curves has modulus equal to zero. From this point of view, the claim~\ref{lem:mod_properties:union_0} of the lemma is worth emphasis as it shows that a countable union of families of curves with zero $X$-modulus has $X$-modulus equal to zero.
%
%
\begin{lemma}
\label{lem:mod_properties}
The modulus satisfies the following properties given that $\Gamma_k$, $k\in\Nbb$, are families of curves in $\Pcal$.
\begin{enumerate}
  \item \label{lem:mod_properties:monotone} If $\Gamma_1 \subset \Gamma_2$, then $\Mod_X(\Gamma_1) \le \Mod_X(\Gamma_2)$.
  \item \label{lem:mod_properties:subadditive} If $X$ is a quasi-normed space with the modulus of concavity $c\ge 1$, then $\Mod_X$ is $\sigma$-quasi-additive, i.e.,
  \[
    \Mod_X \biggl( \bigcup_{k=1}^\infty \Gamma_k \biggr) \le \sum_{k=1}^\infty c^k \Mod_X(\Gamma_k).
  \]
  In particular, if $X$ is a normed space, then $\Mod_X$ is $\sigma$-subadditive.
  \item \label{lem:mod_properties:union_0}
    If $\Mod_X(\Gamma_k) = 0$ for every $k\in\Nbb$, then $\Mod_X(\bigcup_{k=1}^\infty \Gamma_k ) = 0$.  
  \item \label{lem:mod_properties:subcurve_family}
    If for every curve $\gamma_1\in\Gamma_1$ there is a subcurve $\gamma_2\in\Gamma_2$ of $\gamma_1$, then $\Mod_X(\Gamma_1) \le \Mod_X(\Gamma_2).$
\end{enumerate}
\end{lemma}
We shall see in the proof that~\ref{lem:mod_properties:subcurve_family} says, roughly speaking, that the longer the curves in $\Gamma$ are, the smaller $\Mod_X(\Gamma)$ is.
\begin{proof}
\ref{lem:mod_properties:monotone} The infimum in the definition of $\Mod_X(\Gamma_1)$ is taken over a larger set of functions than in the definition of $\Mod_X(\Gamma_2)$.

\ref{lem:mod_properties:subadditive} Let $\eps > 0$. For each $k\in\Nbb$, we can find a non-negative Borel function $\rho_k$ such that $\Mod_X(\Gamma_k) \le \|\rho_k\|_X \le \Mod_X(\Gamma_k) + (2c)^{-k} \eps$ while $\int_\gamma \rho_k \,ds\ge 1$ whenever $\gamma\in\Gamma_k$. Let $\rho = \sup_{k\ge1} \rho_k$. Then, $\int_\gamma \rho \,ds\ge 1$ for every curve $\gamma \in \bigcup_{k=1}^\infty \Gamma_k$, and
\begin{align*}
  \Mod_X\biggl(\bigcup_{k=1}^\infty \Gamma_k\biggr) & \le \|\rho \|_X = \biggl\|\sup_{k\ge1} \rho_k\biggr\|_X \le \biggl\| \sum_{k=1}^\infty \rho_k \biggr\|_X \le \sum_{k=1}^\infty c^k \| \rho_k\|_X \\
  & \le \sum_{k=1}^\infty \bigl(c^k  \Mod_X(\Gamma_k) + 2^{-k}\eps\bigr) = \eps + \sum_{k=1}^\infty c^k  \Mod_X(\Gamma_k)\,.
\end{align*}
Letting now $\eps\to0$ finishes the proof of~\ref{lem:mod_properties:subadditive}.

\ref{lem:mod_properties:union_0} The claim follows immediately from~\ref{lem:mod_properties:subadditive}.

\ref{lem:mod_properties:subcurve_family} Let $\eps>0$. Then, we can find a function $\rho \in X$ such that $\int_\gamma \rho\,ds\ge 1$ for every curve $\gamma\in\Gamma_2$ and $\|\rho\|_X \le \Mod_X(\Gamma_2) + \eps$. For every curve $\gamma_1 \in \Gamma_1$ we can find a subcurve $\gamma_2\in\Gamma_2$, and thus, $\int_{\gamma_1} \rho \, ds\ge  \int_{\gamma_2} \rho \, ds \ge 1$. Consequently, $\Mod_X(\Gamma_1) \le \|\rho\|_X \le \Mod_X(\Gamma_2)+ \eps$. Letting $\eps\to 0$ finishes the proof.
\end{proof}
\begin{remark}
If the $r$th power of $\| \cdot \|_X$ is subadditive for some $r\in (0, \infty)$, then the $r$th power of $\Mod_X$ is $\sigma$-subadditive, which can be shown along the same lines as Lemma~\ref{lem:mod_properties}\,\ref{lem:mod_properties:subadditive}. Consequently, $\Mod_X(\cdot)^r$ is an outer measure on $\Gamma(\Pcal)$.
\end{remark}
%
%
\begin{proposition}
\label{pro:meas_f_has_borel_repr}
If $f: \Pcal \to \overline{\Rbb}$ is measurable, then there exist Borel functions $f_1, f_2: \Pcal \to \overline{\Rbb}$ such that $f_1 \le f \le f_2$ and $f_1 = f_2$ a.e.
\end{proposition}
A proof can be found in Bj\"{o}rn and Bj\"{o}rn~\cite[Proposition 1.2]{BjoBjo}.

As already mentioned, it is whether the $X$-modulus is zero or not that is important to all our arguments based on the notion of $X$-modulus. Therefore, we establish a couple of characterizations equivalent to the condition $\Mod_X(\Gamma) = 0$.
%
%
\begin{proposition}
\label{pro:mod0_equiv}
Let $x\in\Pcal$ and let $\Gamma$ be a family of curves in $\Pcal$. The following are equivalent:
\begin{enumerate}
  \item \label{pro:mod0_equiv.a} $\Mod_X(\Gamma) = 0;$
  \item \label{pro:mod0_equiv.b} there is a non-negative Borel function $\rho\in X$ such that $\int_\gamma \rho\,ds = \infty$ for all curves $\gamma\in\Gamma;$
  \item \label{pro:mod0_equiv.c} there is a non-negative measurable function $\rho$ such that $\rho \chi_{B(x, r)}\in X$  for all radii $r>0$, and such that $\int_\gamma \rho\,ds = \infty$ for all curves $\gamma\in\Gamma$.
\end{enumerate}
\end{proposition}
\begin{proof}
\ref{pro:mod0_equiv.a} $\Rightarrow$~\ref{pro:mod0_equiv.b} For every $n\in\Nbb$ there is a non-negative Borel function $\rho_n\in X$ such that
\[
  \int_\gamma \rho_n\,ds \ge 1 \mbox{ for all }\gamma\in\Gamma, \quad \mbox{and}\quad
  \|\rho_n\|_X \le (2c)^{-n},
\]
where $c\ge 1$ is the modulus of concavity appearing in the triangle inequality in~\ref{df:qBFL.quasinorm} and consequently in the Riesz--Fischer property~\ref{df:qBFL.RF} of $X$. Let $\rho = \sum_{n=1}^\infty \rho_n \in X$. Then, $\int_\gamma \rho \,ds = \infty$ for all $\gamma\in\Gamma$.

\ref{pro:mod0_equiv.b} $\Rightarrow$~\ref{pro:mod0_equiv.c} This implication follows from the lattice property~\ref{df:BFL.latticeprop} of the function space~$X$.

\ref{pro:mod0_equiv.c} $\Rightarrow$~\ref{pro:mod0_equiv.a} Due to Proposition~\ref{pro:meas_f_has_borel_repr}, there is a non-negative Borel function $\tilde{\rho}\ge\rho$ such that $\tilde{\rho} = \rho$ a.e. Let
\[
  \rho_n = \frac{1}{n} \sum_{k=1}^\infty \frac{\tilde{\rho} \chi_{B(x, k)}}{(2c)^k \|\tilde{\rho} \chi_{B(x, k)}\|_X+1}\,,
\]
where $c\ge 1$ retains its meaning as previously, while $n\in\Nbb$. Then, $\|\rho_n\|_X \le 1/n$. If now $\gamma\in\Gamma$, then it is contained within a ball $B(x,k)$ for some $k\in\Nbb$ as the range of the curve $\gamma$ is compact. Therefore, $\int_\gamma \rho_n\,ds =\infty \ge 1$. Hence, $\Mod_X (\Gamma) \le \|\rho_n\|_X \to 0$, as $n\to\infty$.
\end{proof}
%
%
\begin{lemma}
\label{lem:mod_G+_0_meas}
Assume that $\mu(E) = 0$, then $\Mod_X(\Gamma_E^+) = 0$.
\end{lemma}
\begin{proof}
Let $F\supset E$ be a Borel set of zero measure, then $\Gamma_F^+ \supset \Gamma_E^+$. Let $\rho=\infty$ on $F$, outside of which let $\rho$ be zero. Every curve $\gamma\in\Gamma_F^+$ satisfies $\lambda^1(\gamma^{-1}(F)) > 0$ while the set $F\cap \gamma$ is Borel, hence $\int_\gamma \rho\,ds$ is well defined and attains the value $\infty$. Finally, $\Mod_X(\Gamma^+_E) \le \Mod_X(\Gamma^+_F) \le \| \rho \|_X = 0$ since $\rho = 0$ a.e.
\end{proof}
The following lemma shows that we may modify a non-negative measurable function on a set of measure zero while the value of the path integral of this function over a curve $\gamma$ remains the same for $\Mod_X$-a.e.\@ curve $\gamma$.
%
%
\begin{lemma}
\label{lem:mod_ae_equivalence}
Let $g_1$ and $g_2$ be non-negative measurable functions such that  $g_1=g_2$ a.e. Then
\[
  \int_\gamma g_1\,ds = \int_\gamma g_2\,ds \quad \mbox{for $\Mod_X$-a.e.\@ curve $\gamma$.}
\]
In particular, $\int_\gamma g_1\,ds$ is well defined and has a value in $[0, \infty]$ for $\Mod_X$-a.e.\@ curve $\gamma$.
\end{lemma}
\begin{proof}
According to Proposition~\ref{pro:meas_f_has_borel_repr}, there is a non-negative Borel function $g$ such that $g_1 = g = g_2$ a.e. Let $E = \{ x\in\Pcal\colon g_1(x) \neq g(x)\}$. As $g$ is Borel, the integral $\int_\gamma g\,ds$ is well defined for all curves $\gamma$. For curves $\gamma \notin \Gamma_E^+$,
\[
  \int_\gamma g_1\,ds = \int_\gamma g\,ds.
\]
Since $\mu(E)=0$, we have $\Mod_X(\Gamma_E^+) = 0$ by Lemma~\ref{lem:mod_G+_0_meas}. Hence, equality holds for $\Mod_X$-a.e.\@ curve $\gamma$.

A similar argument shows that the equality $\int_\gamma g_2 \,ds = \int_\gamma g\,ds$ holds for $\Mod_X$-a.e.\@ curve $\gamma$. Lemma~\ref{lem:mod_properties}\,\ref{lem:mod_properties:union_0} then finishes the proof.
\end{proof}
%
%
%
%
\section{Weak upper gradients}
\label{sec:wug}
%
%
The set of upper gradients is not a closed subset of $X$, which we have already seen in Example~\ref{exa:trivial_N1p} and similar examples can be provided for non-trivial Newtonian spaces, as well. Another drawback of upper gradients is that they are required to be Borel functions. We can, however, relax the conditions in Definition~\ref{df:ug} to replace the upper gradients with a more flexible set of functions, following the ideas of Koskela and MacManus in~\cite{KosMac}.
%
%
\begin{definition}
A non-negative measurable function $g$ on $\Pcal$ is an \emph{$X$-weak upper gradient} of an extended real-valued function $u$ on $\Pcal$ if 
\begin{equation}
  \label{eq:wug_ineq_def}
  |u(\gamma(0)) - u(\gamma(l_\gamma))| \le \int_\gamma g\,ds
\end{equation}
for $\Mod_X$-a.e.\@ rectifiable curve $\gamma: [0, l_\gamma]\to\Pcal$.
\end{definition}
%
%
\begin{remark}
By Lemma~\ref{lem:mod_ae_equivalence}, the path integral \eqref{eq:wug_ineq_def} is well defined for $\Mod_X$-a.e.\@ curve $\gamma$. Applying Proposition~\ref{pro:meas_f_has_borel_repr} as well, one can see that if a measurable function $g$ is an $X$-weak upper gradient of $u$, then there exists a non-negative Borel function $g'$, which obeys $g'=g$ a.e., and $g'$ is an $X$-weak upper gradient of $u$.
\end{remark}
\begin{remark}
Lemma~\ref{lem:mod_ae_equivalence} also shows that we may modify an $X$-weak upper gradient of a function on a set of measure zero to obtain another $X$-weak upper gradient of the same function. The following example shows that the corresponding claim for upper gradients is false.
\end{remark}
\begin{example}
Let $u: \Rbb^2 \to \Rbb$ be given by $u(x) = |x|$. Then, $g=1$ is an upper gradient of $u$. Let $M$ be the image of a (rectifiable) curve in $\Rbb^2$ of positive length. Then $g' = 1 - \chi_M$ is not an upper gradient of $u$, but $g=g'$ a.e., whence it is an $X$-weak upper gradient of $u$.
\end{example}
Similarly as in the case of upper gradients, we can determine an $X$-weak upper gradient of a linear combination of functions whose $X$-weak upper gradients are known.
%
%
\begin{lemma}
\label{lem:wug.linearity}
Let $g$ and $h$ be $X$-weak upper gradients of $u$ and $v$, respectively, and $a\in\Rbb$. Then, $|a|g$ and $g + h$ are $X$-weak upper gradients of $au$ and $u + v$, respectively.
\end{lemma}
\begin{proof}
Let $\Gamma_1$ be the family of exceptional curves for $g$ with $u$, and $\Gamma_2$ for $h$ with $v$. Then, $\Gamma_1$ is the exceptional family for $|a|g$ with $au$, whereas $\Gamma_1 \cup \Gamma_2$ is the exceptional family for $g+h$ with $u+v$. Lemma~\ref{lem:mod_properties}~\ref{lem:mod_properties:union_0} now ensures that $\Mod_X(\Gamma_1 \cup \Gamma_2)=0$.
\end{proof}
The following lemma shows that $X$-weak upper gradients of a given function can be approximated by its upper gradients with arbitrarily small distance in $X$. Note that here we do not require that the approximated $X$-weak upper gradient lies in $X$.
%
%
\begin{lemma}
\label{lem:ug-approx-wug}
Let $g$ be an $X$-weak upper gradient of $u$. Then, there exist $\rho_k\in X$ such that $g+\rho_k$ is an upper gradient of $u$ for every $k\in\Nbb$ and $\|\rho_k\|_X\to 0$ as $k\to \infty$. In fact, there is $\rho \in X$ such that we may choose $\rho_k = \rho/k$ for every $k\in\Nbb$.
\end{lemma}
\begin{proof}
We can find a non-negative Borel function $g'$ such that $g'=g$ a.e. Lemma~\ref{lem:mod_ae_equivalence} shows that $g'$ is an $X$-weak upper gradient of $u$ as well.
Let $\Gamma$ consist of those curves $\gamma: [0, l_\gamma] \to \Pcal$ such that
\[
  |u(\gamma(0)) - u(\gamma(l_\gamma))| \nleq  \int_\gamma g'\,ds.
\]
Therefore, $\Mod_X(\Gamma) = 0$, and hence, by Proposition~\ref{pro:mod0_equiv}, there is a non-negative Borel function $\rho\in X$ such that $\int_\gamma \rho\,ds = \infty$ for all $\gamma \in \Gamma$. Due to Borel regularity of the measure on $\Pcal$, there is a Borel set $M$ of zero measure such that it contains the set $\{x\in\Pcal: g(x) \neq g'(x)\}$. Finally, let
\[
  \rho_k(x) = \begin{cases}
                 \dfrac{\rho(x)}{k} & \mbox{for $x\in \Pcal \setminus M$}, \\
                 \,\infty & \mbox{for $x\in M$}.
              \end{cases}
\]
Then, $g+ \rho_k = g' + \rho_k$ is a Borel function. It is also an upper gradient of $u$ and $\| \rho_k \|_X = \|\rho\|_X / k \to 0$ as $k\to\infty$.
\end{proof}
Consequently, we could have defined the $\NX$ (quasi)seminorm using $X$-weak upper gradients instead of upper gradients as is proven in the following corollary.
%
%
\begin{corollary}
\label{cor:def_N1X_via_WUGs}
Let $u\in \Mcal(\Pcal, \mu)$. Then,
\[
  \|u\|_\NX = \|u\|_X + \inf_g \|g\|_X\,,
\]
where the infimum is taken over all $X$-weak upper gradients $g$ of $u$.
\end{corollary}
\begin{proof}
Let $m$ be the infimum as in the claim. Let $\widetilde{m}= \inf_g \|g\|_X$, where the infimum is taken only over all upper gradients $g$ of $u$. We immediately obtain that $m\le \widetilde{m}$. If $m=\infty$, then obviously $m=\widetilde{m}$. Suppose now that $m<\infty$. Then there exists a sequence of $X$-weak upper gradients $\{g_k\}_{k=1}^\infty$ of $u$ such that $\|g_k\|_X \to m$ as $k\to\infty$. By Lemma~\ref{lem:ug-approx-wug}, there are upper gradients $\tilde{g}_k$ of $u$ such that $\|g_k - \tilde{g}_k\|_X < 1/k$ for all $k\in\Nbb$. Therefore, $\|\tilde{g}_k\|_X \to m$ as $k\to\infty$, and hence $\widetilde{m}\le m$. This fact leads to equality $\|u\|_X + \widetilde{m} = \|u\|_X + m$, which finishes the proof.
\end{proof}
%
%
\begin{definition}
A Borel function $g: \Pcal\to [0, \infty]$ is called an \emph{upper gradient of $u$ along a curve $\gamma$} if it satisfies inequality \eqref{eq:ug_def} for every subcurve $\gamma'$ of $\gamma$.
\end{definition}
%
%
\begin{corollary}
\label{cor:not_ug_along_curves}
If $g$ is an $X$-weak upper gradient of $u$ on $\Pcal$ and
\[
  \Gamma = \{\gamma \in \Gamma(\Pcal)\colon \mbox{$g$ is not an upper gradient of $u$ along $\gamma$}\},
\]
then $\Mod_X(\Gamma) = 0$.
\end{corollary}
\begin{proof}
Let $\Gamma'$ consist of those curves $\gamma': [0, l_{\gamma'}] \to \Pcal$ for which
\[
  |u({\gamma'}(0)) - u(\gamma'(l_{\gamma'}))| \nleq \int_{\gamma'} g\,ds.
\]
Then, $\Mod_X(\Gamma') = 0$ by the definition of weak upper gradient. Moreover, each curve $\gamma \in \Gamma$ has a subcurve $\gamma' \in \Gamma'$, whence $\Mod_X(\Gamma) \le \Mod_X(\Gamma') = 0$ by Lemma~\ref{lem:mod_properties}\,\ref{lem:mod_properties:subcurve_family}.
\end{proof}
Having a fixed set, we shall find a close relation between its negligibility in terms of $X$-capacity and and in terms of $X$-modulus of the family of curves intersecting it.
%
%
\begin{proposition}
\label{pro:cap0_meas0_mod0}
Let $E\subset \Pcal$. Then, $C_X(E) = 0$ if and only if $\mu(E) = \Mod_X(\Gamma_E) = 0$.
\end{proposition}
\begin{proof}
Assume first that $\mu(E) = \Mod_X(\Gamma_E) = 0$. Let $u = \chi_E$. Then, $g\equiv0$ is an $X$-weak upper gradient of $u$ since $u\equiv0$ on all curves outside $\Gamma_E$, i.e., on $\Mod_X$-a.e.\@ curve. Hence, $C_X(E) \le \|u\|_{\NX} = 0$ as $u=0$ a.e.

Suppose now that $C_X(E) = 0$. It follows from Theorem~\ref{thm:CX-properties}~\ref{itm:CX-vs-measure} that $\mu(E) = 0$. Let $\{u_j\}_{j=1}^\infty$ be a sequence of functions in $\NX$ with their respective upper gradients $g_j$ such that $\chi_E \le u_j \le 1$, while $\|u_j\|_{\NX} < (2c)^{-j}$ as well as $\|g_j\|_{X} < (2c)^{-j}$ for all $j\in\Nbb$. Let $u=\sum_{j=1}^\infty u_j \in X$ and $g=\sum_{j=1}^\infty g_j \in X$.
Let
\begin{align*}
  F &= \{x\in \Pcal: u(x) = \infty\} \supset E, \\
  \Gamma_1 & = \Bigl\{\gamma\in\Gamma(\Pcal): \int_\gamma g\,ds = \infty\Bigr\}, \\
  \Gamma_2 & = \{\gamma\in\Gamma(\Pcal): \gamma \subset F\} \subset \Bigl\{\gamma\in\Gamma(\Pcal): \int_\gamma u\,ds = \infty\Bigr\}.
\end{align*}
Now, $\Mod_X(\Gamma_1) = 0$ by Proposition~\ref{pro:mod0_equiv} as $g\in X$. Since $u\in X$, we similarly obtain that $\Mod_X(\Gamma_2) = 0$. What remains to be proven is that $\Gamma_F \subset \Gamma_1 \cup \Gamma_2$. Therefore, let $\gamma \in \Gamma(\Pcal)\setminus (\Gamma_1 \cup \Gamma_2)$. Then, there is $x\in \gamma \setminus F$, i.e., $u(x) < \infty$. For every $z \in \gamma$ we have
\begin{align*}
  u(z) & = \sum_{j=1}^\infty u_j(z) \le \sum_{j=1}^\infty u_j(x) + \sum_{j=1}^\infty |u_j(z)-u_j(x)| \\
  & \le u(x) + \sum_{j=1}^\infty \int_\gamma g_j\,ds = u(x) + \int_\gamma g\,ds < \infty,
\end{align*}
whence $\gamma \cap F = \emptyset$. Finally, $\Mod_X(\Gamma_E) \le \Mod_X(\Gamma_F) \le \Mod_X(\Gamma_1\cup \Gamma_2) = 0$.
\end{proof}
Previously, we have seen that modifying an $X$-weak upper gradient on a set of measure zero preserves its properties. The following corollary shows that modifying a function on a set of $X$-capacity zero retains its $X$-weak upper gradients.
%
%
\begin{corollary}
\label{cor:equal_qe_same_wug}
If $u=v$ q.e.\@ and $g$ is an $X$-weak upper gradient of $u$, then $g$ is also an $X$-weak upper gradient of $v$.
\end{corollary}
\begin{proof}
Let $E=\{x\in \Pcal: u(x)\neq v(x)\}$. Then, $C_X(E)=0$ and hence $\Mod_X(\Gamma_E) = 0$. Consequently, $u=v$ along $\Mod_X$-a.e.\@ curve, which implies that $g$ is an upper gradient of $v$ along $\Mod_X$-a.e.\@ curve.
\end{proof}
The next corollary provides us with an alternative definition of an $X$-weak upper gradient of an a.e.\@ finite function. It shows that it does not really matter how we interpret the inequality \eqref{eq:wug_ineq_def} when the left-hand side is $|(\pm \infty) - (\pm \infty)|$. 
%
%
\begin{proposition}
\label{pro:alt-def-wug}
Let $u: \Pcal \to \overline{\Rbb}$ be a function which is finite a.e.\@ and assume that $g\ge0$ is such that for $\Mod_X$-a.e.\@ curve $\gamma: [0, l_\gamma] \to \Pcal$ it is true that either
\begin{equation}
  \label{eq:alt-def-wug}
  |u(\gamma(0))| = |u(\gamma(l_\gamma))| = \infty
  \quad \mbox{or}
  \quad
  |u(\gamma(0)) - u(\gamma(l_\gamma))| \le \int_\gamma g\,ds.
\end{equation}
Then, $g$ is an $X$-weak upper gradient of $u$.
\end{proposition}
Such a characterization of weak upper gradients was originally given by Bj\"{o}rn, Bj\"{o}rn, and Parviainen in~\cite{BjoBjoPar} for the $L^p$ case.
\begin{proof}
Let $\Gamma$ be the set of all curves that have a subcurve for which \eqref{eq:alt-def-wug} does not hold. Then, $\Mod_X(\Gamma)=0$ by Lemma~\ref{lem:mod_properties}\,\ref{lem:mod_properties:subcurve_family}. Let $E = \{x\in \Pcal: |u(x)| = \infty\}$ and $\Gamma^*_E = \{\gamma\in\Gamma(\Pcal): \gamma \subset E\}$. We have $\Mod_X(\Gamma^*_E) \le \Mod_X(\Gamma_E^+)$, which is equal to zero by Lemma~\ref{lem:mod_G+_0_meas} as $\mu(E)=0$. Let now $\gamma \in \Gamma(\Pcal) \setminus (\Gamma \cup \Gamma_E^*)$. Then, there is $t\in[0, l_\gamma]$ such that $\gamma(t) \notin E$. If $t=0$ or $t=l_\gamma$, then
\[
  |u(\gamma(0)) - u(\gamma(l_\gamma))| \le \int_\gamma g\,ds
\]
by the hypotheses. Otherwise,
\begin{align*}
  |u(\gamma(0)) - u(\gamma(l_\gamma))| & \le |u(\gamma(0)) - u(\gamma(t))| + |u(\gamma(t)) - u(\gamma(l_\gamma))| \\
  & \le \int_{\gamma\mid_{[0, t]}} g\,ds + \int_{\gamma\mid_{[t, l_\gamma]}} g\,ds = \int_\gamma g\, ds
\end{align*}
because the second alternative in \eqref{eq:alt-def-wug} holds for both $\gamma\mid_{[0, t]}$ and $\gamma\mid_{[t, l_\gamma]}$. Therefore, $g$ is an $X$-weak upper gradient of $u$ since $\Mod_X(\Gamma \cup \Gamma_E^*) = 0$ due to Lemma~\ref{lem:mod_properties}\,\ref{lem:mod_properties:union_0}.
\end{proof}
%
%
%
%
\section{Absolute continuity along curves}
\label{sec:acc}
%
%
Due to work of Beppo Levi~\cite[Section 3]{Lev}, which goes as far back as 1906, it is well known that the classical Sobolev space $W^{1,p}(\Rbb^n)$ can be characterized as the space of functions in $L^p(\Rbb^n)$, which have a representative that is ACL, i.e., absolutely continuous on almost every line parallel to the coordinate axes and whose (classical) partial derivatives belong to $L^p(\Rbb^n)$ as well, see e.g.\@ Ziemer~\cite[Theorem 2.1.4]{Zie}.

On metric spaces there are no distinctively preferable lines; however, we can study the (absolute) continuity of Newtonian functions on curves. Since Levi's characterization allowed a certain number of exceptional lines, it is natural to expect that there are exceptional curves in our setting as well.
%
%
\begin{definition}
\label{df:DX}
A measurable function belongs to the \emph{Dirichlet space} $DX$ if it has an upper gradient in $X$.
\end{definition}
%
%
\begin{remark}
Due to Lemma~\ref{lem:ug-approx-wug}, we can equivalently define the Dirichlet space $DX$ by requiring existence of an $X$-weak upper gradient lying in $X$. We can easily see that $DX$ is a linear space, with $\NX$ as a subspace. Moreover, the proof of Theorem~\ref{thm:N1X-lattice} shows that $DX$ is a lattice.
\end{remark}
%
%
\begin{definition}
\label{df:AC}
A function $f: [a, b] \to \Rbb$ is \emph{absolutely continuous} on $[a, b]$, abbreviated as $f\in \AC([a,b])$, if for every $\eps > 0$ there is $\delta>0$ such that
\[
  \sum_{j=1}^n |f(b_j) - f(a_j)| < \eps
\]
for any $n\in\Nbb$ and any $a\le a_1 < b_1 \le a_2 <b_2 \le \cdots \le a_n < b_n \le b$ such that
\[
  \sum_{j=1}^n (b_j - a_j) < \delta.
\]
\end{definition}
\begin{remark}
It can be proven (cf.\@ Rudin~\cite[Theorem 7.20]{Rud}) that $f\in\AC([a, b])$ if and only if $f$ is differentiable a.e.\@ in $[a,b]$, while $f'\in L^1([a,b])$ and
\[
  f(x) = f(a) + \int_a^x f'(t)\,dt \quad \mbox{for all } x\in [a,b].
\]
\end{remark}
\begin{definition}
\label{df:ACCX}
A function $u: \Pcal\to\overline{\Rbb}$ is \emph{absolutely continuous on $\Mod_X$-a.e.\@ curve}, shortened as $u\in \ACC_X(\Pcal)$, if the function $u\circ \gamma: [0, l_\gamma] \to \Rbb$ is absolutely continuous for all curves $\gamma: [0, l_\gamma] \to \Pcal$ except perhaps for a family of curves with zero $X$-modulus.
\end{definition}
%
%
\begin{lemma}
Suppose $\alpha \in \Rbb$ and $w: \Rbb\to \Rbb$ is a Lipschitz function. If $u, v$ are in $\ACC_X(\Pcal)$, then so are $u\pm v$, $\alpha u$, $uv$, $\max\{u,v\}$, $\min\{u,v\}$, $w\circ u$, $|u|$, $u^+$, and $u^-$, where $w(\pm \infty)$ may be defined arbitrarily. In particular, $\ACC_X(\Pcal)$ is a lattice.
\end{lemma}
\begin{proof}
Let $\Gamma_1$ and $\Gamma_2$ be the families of the exceptional curves for $u$ and $v$, respectively. Then, $\Mod_X(\Gamma_1 \cup \Gamma_2) = 0$ by Lemma~\ref{lem:mod_properties}\,\ref{lem:mod_properties:union_0}. Consider a curve $\gamma:[0, l_\gamma] \to \Pcal$, which lies in $\Gamma(\Pcal) \setminus (\Gamma_1 \cup \Gamma_2)$, and let $f = u\circ\gamma$ and $g = v \circ \gamma$. Then, $f,g \in \AC([0, l_\gamma])$ and the absolute continuity of $f + g$, $fg$, $\max\{f,g\}$, and $w\circ f$ on $[0, l_\gamma]$ follows from Lemma~1.58 in Bj\"{o}rn and Bj\"{o}rn~\cite{BjoBjo}. Note that $f$ is a bounded function, so the values of $w$ at the infinities do not have any effect on the function $w\circ f$. Therefore, $u+v$, $uv$, $\max\{u,v\}$, and $w\circ u$ are absolutely continuous along $\gamma$. 

Absolute continuity  of all the remaining functions in the claim along almost every curve follows from the facts which have just been established.
\end{proof}
The following theorem shows that Dirichlet functions, and in particular Newtonian functions, are absolutely continuous on $\Mod_X$-a.e.\@ curve. Such a result can be seen as stronger than the ACL condition for Sobolev functions on $\Rbb^n$.
%
%
\begin{theorem}
\label{thm:DX-subset-ACC}
If $u\in DX$, then $u\in \ACC_X(\Pcal)$.
\end{theorem}
\begin{proof}
Let $g\in X$ be an upper gradient of $u$ and let $\Gamma$ consist of those curves for which $\int_\gamma g\,ds = \infty$. Then, $\Mod_X(\Gamma) = 0$ due to Proposition~\ref{pro:mod0_equiv}.

Let $\gamma: [0, l_\gamma]\to\Pcal$ be a curve in $\Gamma(\Pcal)\setminus \Gamma$. For all $a,b\in[0, l_\gamma]$, $a<b$, we have
\begin{equation}
  \label{eq:u-ACC-gamma}
  |u(\gamma(b)) - u(\gamma(a))| \le \int_{\gamma|_{[a,b]}} g\,ds < \infty,
\end{equation}
and in particular $u(\gamma(a)) \in \Rbb$ for all $a\in [0, l_\gamma]$. Suppose now that $f \coloneq u\circ \gamma$ is not absolutely continuous on $[0, l_\gamma]$. Then, there is an $\eps>0$ such that for every $j\in\Nbb$ there are $0\le a_{j,1} < b_{j,1} \le a_{j,2}< \cdots \le a_{j,n_j} < b_{j,n_j} \le l_\gamma$ such that
\[
  \sum_{i=1}^{n_j} (b_{j,i} - a_{j,i}) < 2^{-j}, \quad \mbox{while} \quad 
  \sum_{i=1}^{n_j} |f(b_{j,i}) - f(a_{j,i})| \ge \eps.
\]
Letting $E_j = \bigcup_{i=1}^{n_j} [a_{j,i}, b_{j,i}]$, we obtain by \eqref{eq:u-ACC-gamma} and the dominated convergence theorem that
\[
  \eps \le \sum_{i=1}^{n_j} |f(b_{j,i}) - f(a_{j,i})| \le \int_{\gamma|_{E_j}} g\,ds \to 0\quad\mbox{as }j\to \infty
\]
since $\lambda^1(E_j) < 2^{-j}$. This contradiction shows that $f$ is necessarily absolutely continuous on $[0,l_\gamma]$, whence $u$ is absolutely continuous on every $\gamma \in \Gamma(\Pcal)\setminus\Gamma$.
\end{proof}
Having a function $u\in\ACC_X(\Pcal)$ and a curve $\gamma$ on which $u$ is absolutely continuous, one can compare the classical derivative of $u\circ \gamma$ with an arbitrary $X$-weak upper gradient. This leads to the following lemma, which provides us with yet another characterization of $X$-weak upper gradients lying in $X$.
%
%
\begin{lemma}
\label{lem:ACC-derivative-vs-ug}
Assume that $u\in\ACC_X(\Pcal)$ and that $g\in X$ is an $X$-weak upper gradient of $u$. Then, for $\Mod_X$-a.e.\@ curve $\gamma: [0, l_\gamma]\to \Pcal$, we have
\begin{equation}
\label{eq:ACC-derivative-vs-ug}
  |(u\circ \gamma)'(t)| \le g(\gamma(t)) \quad\mbox{for a.e.\@ }t\in[0, l_\gamma].
\end{equation}
Conversely, if $g\ge 0$ is measurable, $u\in\ACC_X(\Pcal)$, and \eqref{eq:ACC-derivative-vs-ug} holds for $\Mod_X$-a.e.\@ curve $\gamma: [0, l_\gamma]\to\Pcal$, then $g$ is an $X$-weak upper gradient of $u$.
\end{lemma}
Observe that whereas we need to assume that $g\in X$ in the forward implication (cf.\@ Example~\ref{exa:cantor_infty_gradient} below), it is unnecessary in the converse.
\begin{proof}
Assume first that $u\in\ACC_X(\Pcal)$ and that $g\in X$ is an $X$-weak upper gradient of~$u$. Let $\gamma: [0, l_\gamma]\to\Pcal$ be a curve such that $u\circ \gamma \in \AC([0, l_\gamma])$, $g$ is an upper gradient of $u$ along $\gamma$, and $\int_\gamma g\,ds<\infty$. This holds for $\Mod_X$-a.e.\@ curve.

Now, almost every point $t\in(0,l_\gamma)$ is a Lebesgue point for $g\circ \gamma$ (so that the last equality in \eqref{eq:leb_point} below holds) and simultaneously $u\circ \gamma$ is differentiable at $t$. For such~$t$ we obtain
\begin{align}
  \label{eq:leb_point}
  |(u\circ \gamma)'(t)| & = \lim_{h\to 0} \Bigl| \frac{u(\gamma(t+h)) - u(\gamma(t))}{h}\Bigr|
  \\
  & \le \lim_{h\to 0} \frac{1}{h} \int_t^{t+h} g(\gamma(\tau))\,d\tau = g(\gamma(t)).
  \nonumber
\end{align}

Conversely, assume that $g\ge0$ is measurable, $u\in\ACC_X(\Pcal)$ and \eqref{eq:ACC-derivative-vs-ug} holds for $\Mod_X$-a.e.\@ curve $\gamma: [0, l_\gamma]\to\Pcal$. Let $\gamma: [0, l_\gamma]\to\Pcal$ be a curve on which $u$ is absolutely continuous, \eqref{eq:ACC-derivative-vs-ug} holds, and $\int_\gamma g\,ds$ is well defined. This holds for $\Mod_X$-a.e.\@ curve as can be seen by Lemma~\ref{lem:mod_ae_equivalence}. Then,
\[
  |u(\gamma(0)) - u(\gamma(l_\gamma))| \le \int_0^{l_\gamma}|(u\circ\gamma)'(t)|\,dt \le \int_0^{l_\gamma} g(\gamma(t))\,dt = \int_\gamma g\,ds.
  \qedhere
\]
\end{proof}
\begin{example}
\label{exa:cantor_infty_gradient}
Let $A \subset [0, 1]$ be a Borel set which satisfies $0<\lambda^1(A \cap I)<\lambda^1(I)$ for all non-degenerate intervals $I \subset [0,1]$. Then, $g= \infty \chi_A$ is an upper gradient of any function on $[0,1]$. Let $u(x) = x$ for $x\in [0,1]$. Hence, $(u\circ \gamma)'(t) = 1 \nleq 0 = g(\gamma(t))$ whenever $\gamma(t) \notin A$, which happens for $t$ chosen from a set of positive measure.
\end{example}
As a consequence of the last characterization, we can easily derive the product and chain rule for $X$-weak upper gradients. These rules can be proven analogously as~\cite[Theorems~2.15 and~2.16]{BjoBjo}, whence the proofs are omitted.
\begin{proposition}[Product rule]
Let $u_1, u_2 \in DX$ and let $g_1, g_2 \in X$ be their respective $X$-weak upper gradients. Then, $|u_1| g_2 + |u_2| g_1$ is an $X$-weak upper gradient of $u_1 u_2$.
\end{proposition}
\begin{proposition}[Chain rule]
Let $w: I \to \overline{\Rbb}$ be locally Lipschitz, where $I \subset \overline{\Rbb}$ is an interval. Let $u: \Pcal \to I$ and suppose that $u \in DX$ with an $X$-weak upper gradient $g \in X$. Then, $|w'\circ u| g$ is an $X$-weak upper gradient of $w \circ u$, where we set $w'=0$ wherever undefined.
\end{proposition}
Next, we will investigate how a pair of functions related by pointwise (in)equality a.e.\@ is affected by the fact that these functions belong to $\ACC_X(\Pcal)$, which will eventually lead to a description of the natural equivalence classes for Newtonian functions.
%
%
\begin{proposition}
\label{pro:ACC-eq.ae-eq.qe}
If $u,v \in \ACC_X(\Pcal)$ and $u = v$ a.e., then $u = v$ q.e.
\end{proposition}
\begin{proof}
Let $E = \{ x\in \Pcal: u(x) \neq v(x)\}$. Since $\mu(E) = 0$, we have $\Mod_X(\Gamma_E^+) = 0$ by Lemma~\ref{lem:mod_G+_0_meas}. Let $\Gamma = \{\gamma \in \Gamma(\Pcal) \setminus \Gamma_E^+: u\circ \gamma, v\circ \gamma \in \AC([0, l_\gamma])\}$, which gives $\Mod_X(\Gamma^{\C}) = 0$. Suppose $\gamma \in \Gamma$. Then, $\lambda^1(\gamma^{-1}(E)) = 0$, i.e., $u\circ\gamma=v\circ\gamma$ $\lambda^1$-a.e.\@ on $[0,l_\gamma]$. As both functions $u$ and $v$ are continuous on $\gamma$, we have $u=v$ everywhere on $\gamma$, whence $\gamma \cap E = \emptyset$. Consequently, $\Gamma_E \subset \Gamma^\C$ and $\Mod_X(\Gamma_E) \le \Mod_X(\Gamma^\C) = 0$. We can now conclude that $C_X(E) = 0$ due to Proposition~\ref{pro:cap0_meas0_mod0}.
\end{proof}
%
%
\begin{corollary}
\label{cor:ACC-ineq.ae-ineq.qe}
If $u,v \in \ACC_X(\Pcal)$ and $u \ge v$ a.e., then $u \ge v$ q.e.
\end{corollary}
\begin{proof}
Let $w = \min\{u, v\}$. Then, $w\in \ACC_X(\Pcal)$, and $u\ge w$ everywhere in $\Pcal$ while $v=w$ a.e. We can see that $u\ge w=v$ q.e.\@ due to Proposition~\ref{pro:ACC-eq.ae-eq.qe}.
\end{proof}
%
%
\begin{corollary}
\label{cor:N1X-eq.ae-eq.qe}
Let $u,v \in \NX$. If $u \ge v$ a.e., then $u \ge v$ q.e. Furthermore, if $u = v$ a.e., then $u = v$ q.e.
\end{corollary}
\begin{proof}
Theorem~\ref{thm:DX-subset-ACC} shows that both $u$ and $v$ are absolutely continuous on $\Mod_X$-a.e.\@ curve. Corollary~\ref{cor:ACC-ineq.ae-ineq.qe} and Proposition~\ref{pro:ACC-eq.ae-eq.qe}, respectively, finish the proof.
\end{proof}
The following results prove that the equivalence classes in $\NtX$ are up to sets of $X$-capacity zero. Therefore, we can see that there is an actual difference between Newtonian spaces and Sobolev spaces on $\Rbb^n$ as the latter are defined with equivalence classes up to sets of measure zero. Remember that the capacity in general allows a finer distinction of sets of measure zero.
%
%
\begin{proposition}
\label{pro:N1X-zero-functions}
Let $u: \Pcal \to \overline{\Rbb}$ be a measurable function. Then, $\|u\|_{\NX} = 0$ if and only if $u=0$ q.e.
\end{proposition}
\begin{proof}
Assume first that $\|u\|_{\NX} = 0$. Then, $u\in \NX$ and $u=0$ a.e.\@ as $\|u\|_X = 0$. Thus, $u=0$ q.e.\@ by Corollary~\ref{cor:N1X-eq.ae-eq.qe}.

Assume, on the other hand, that $E=\{x\in\Pcal: u(x) \neq 0\}$ satisfies $C_X(E) = 0$. Then, $\mu(E) = 0$ by Proposition~\ref{pro:cap0_meas0_mod0} while $0$ is an $X$-weak upper gradient of $u$ by Corollary~\ref{cor:equal_qe_same_wug}. Therefore, $\|u\|_{\NX} \le \|u\|_X + \|0\|_X = 0$.
\end{proof}
%
%
\begin{corollary}
\label{cor:nat.eq.classes}
The equivalence classes in $\NtX$ are given by equality up to sets of capacity zero.
\end{corollary}
\begin{proof}
  The claim follows from the definition of $u\sim v$ by condition $\|u-v\|_{\NX} = 0$, which holds if and only if $u-v = 0$ q.e.\@ as has been shown in Proposition~\ref{pro:N1X-zero-functions}.
\end{proof}
Next, we introduce a space of equivalence classes described by equality a.e.\@ with a Newtonian function, which is a closer counterpart of classical Sobolev spaces.
%
%
\begin{definition}
Let us introduce the space of equivalence classes given by a.e.\@ equality to Newtonian functions, i.e.,
\[ 
  \NhX = \{u \in X: \mbox{ there is $v\in \NX$ such that $u=v$ a.e.}\},
\]
with the norm induced by $\NX$ (so that $\|u\|_{\NhX} = \|v\|_{\NX}$, whenever $u=v$ a.e., while $u\in \NhX$ and $v \in \NX$). Observe that the norm is well defined due to Corollaries~\ref{cor:N1X-eq.ae-eq.qe} and~\ref{cor:nat.eq.classes}.
\end{definition}
The following proposition quantifies the difference between equivalence classes of functions in $\NhX$ and functions in $\NX$. Roughly speaking, $\NX$ consists only of the ``good'' representatives of classes in $\NhX$.
%
%
\begin{proposition}
Let $u\in \NhX$, then $u\in \NX$ if and only if $u\in \ACC_X(\Pcal)$.
\end{proposition}
\begin{proof}
  Theorem~\ref{thm:DX-subset-ACC} gives the necessity.
  
  Consider now $u\in \NhX \cap \ACC_X(\Pcal)$. Then, there is $v\in \NX \subset \ACC_X(\Pcal)$ such that $u=v$ a.e. Proposition~\ref{pro:ACC-eq.ae-eq.qe} implies that $u=v$ q.e., whence $\|u-v\|_{\NX} = 0$ by Proposition~\ref{pro:N1X-zero-functions}, and hence $u \in \NX$.
\end{proof}
\begin{example}
If we consider $X\subset L^1_\loc(\Rbb)$, then $\chi_\Qbb \notin \NX$ as it does not have any upper gradient in $L^1_\loc(\Rbb)$. On the other hand, $\chi_\Qbb = 0$ a.e.\@ on $\Rbb$, whence $\chi_\Qbb \in \NhX$.
\end{example}
In the Euclidean case, we have $\widehat{N}^1 L^p = W^{1,p}$. Indeed, Ohtsuka has shown in~\cite[Sections~4.3 and~4.4]{Oht} that an $L^p$ function lies in the Sobolev space $W^{1,p}$ if and only if it has an $\ACC_{L^p}$ representative whose gradient is integrable to the $p$th power. The description of zero $L^p$-modulus of a family of curves in his text corresponds to our Proposition~\ref{pro:mod0_equiv}\,\ref{pro:mod0_equiv.b}. 
%
%
%
%
\section{Completeness of Newtonian spaces}
\label{sec:banach}
%
%
We shall further see that the general setting of quasi-Banach function lattices suffices to prove that Newtonian spaces are in fact complete. The proof relies heavily on the fact that the equivalence classes in $\NtX$ are given by equality up to sets of capacity zero.
%
%
\begin{theorem}
\label{thm:N1X-complete}
The Newtonian space $\NtX$ is complete.
\end{theorem}
\begin{proof}
Recall that $\NtX = \NX / \mathord{\sim} = \NX / \mathord{=}_{\mathrm{q.e.}}$ as has been proven in Corollary~\ref{cor:nat.eq.classes}. Let $\{u_j\}_{j=1}^\infty$ be a Cauchy sequence in $\NX$. Without loss of generality we may assume (by passing to a subsequence if necessary) that $\|u_{j+1}-u_j\|_{\NX} < (4c)^{-j}$, where $c\ge1$ is the modulus of concavity of $\NX$. Let
\[
  E_j = \{ x\in \Pcal: |u_{j+1}(x) - u_j(x)| > 2^{-j}\}.
\]
Consequently, we can estimate $C_X(E_j) \le \|2^j (u_{j+1} - u_j)\|_{\NX} \le 2^j (4c)^{-j} = (2c)^{-j}$. Let $F=\limsup_{j\to\infty} E_j$, i.e.,
\[
  F= \bigcap_{k=1}^\infty F_k, \quad \mbox{where } F_k = \bigcup_{j=k}^\infty E_j.
\]
Hence, $C_X(F_k) \le \sum_{j=k}^\infty c^{j-k+1} C_X(E_j) \le \sum_{j=k}^\infty c^{j-k+1} (2c)^{-j} = (2c)^{1-k}$ by Theorem~\ref{thm:CX-properties}. Therefore, 
$C_X(F)=0$ as $C_X(F) \le C_X(F_k)$ for all $k\in\Nbb$. Let $x\in \Pcal \setminus F$, then $x\in \Pcal \setminus F_m$ for some $m$ whence $|u_{j+1}(x) - u_j(x)| \le 2^{-j}$ for all $j\ge m$ and $\{u_j(x)\}_{j=1}^\infty$ forms a Cauchy sequence in $\Rbb$. Hence, we can define
\[
  u(x) = \begin{cases}
           \lim_{j\to\infty} u_j(x) & \mbox{for } x\in \Pcal \setminus F, \\
           0 & \mbox{for } x \in F.
         \end{cases}
\]
Observe that $u(x)$ is defined as the limit for q.e.\@ $x\in\Pcal$. Moreover, we have
\begin{equation}
  \label{eq:pointwise_limit}
  \lim_{j\to\infty} u_j(x) = u_k(x) + \sum_{j=k}^\infty (u_{j+1}(x) - u_j(x)), \quad x\in\Pcal \setminus \widetilde{F},
\end{equation}
where $\widetilde{F} = F \cup \bigcup_{j=1}^\infty \{x\in \Pcal: |u_j(x)| = \infty\}$ while $k\in \Nbb$ may be chosen arbitrarily. Proposition~\ref{pro:N1X-finite-qe} and Theorem~\ref{thm:CX-properties} yield that $C_X(\widetilde{F})=0$, and hence $\mu(\widetilde{F})=0$. Thus, $\|u-u_k\|_X \le \sum_{j=k}^\infty (4c)^{-j} = (4c)^{1-k}/(4c-1)$, so $u_k \to u$ in $X$ as $k\to\infty$. Since $C_X(\widetilde{F}) = 0$, $\Mod_X$-a.e.\@ curve in $\Pcal$ has an empty intersection with $\widetilde{F}$ by Pro\-po\-si\-tion~\ref{pro:cap0_meas0_mod0}. Let $\gamma$ be one such curve, connecting $x=\gamma(0)$ and $z=\gamma(l_\gamma)$. As $u$ is defined by \eqref{eq:pointwise_limit} on $\gamma$, we have
\begin{align*}
  |(u-u_k)(x) - (u-u_k)(z)| & \le \sum_{j=k}^\infty | (u_{j+1}-u_j)(x) - (u_{j+1}-u_j)(z)| \\
  & \le \sum_{j=k}^\infty \int_\gamma g_j\,ds = \int_\gamma \sum_{j=k}^\infty g_j\,ds,
\end{align*}
where $g_j$ is an upper gradient of $u_{j+1}-u_j$ such that $\|g_j\|_X < (2c)^{-j}$. Therefore, $\tilde{g}_k = \sum_{j=k}^\infty g_j$ is an $X$-weak upper gradient of $u-u_k$, and $\|\tilde{g}_k\|_X \le \sum_{j=k}^\infty c^{j-k+1} (2c)^{-j} = (2c)^{1-k}$. Finally, it follows that
\[
  \|u-u_k\|_{\NX} \le \|u-u_k\|_X + \|\tilde{g}_k\|_X \le \frac{(4c)^{1-k}}{4c-1} + (2c)^{1-k} \to 0 \quad\mbox{as }k\to\infty.
\qedhere
\]
\end{proof}
Finally, we will investigate what consequences the convergence in $\NX$ has on pointwise and uniform convergence of a sequence of functions. A Egorov-type theorem can be considered contained in the following corollary.
%
%
\begin{corollary}
\label{cor:subseq_converges_pointwise}
Assume that $u_j \to u$ in $\NX$ as $j\to \infty$. Then, there is a subsequence which converges to $u$ pointwise q.e. Moreover, for every $\eps>0$ there is a set $E$ with $C_X(E)<\eps$, such that the subsequence converges uniformly to $u$ outside of $E$. If all functions $u_j$ are continuous, then there is an open set $G$ with $C_X(G)<\eps$, such that the subsequence converges uniformly outside of $G$ (not necessarily to $u$, though).
\end{corollary}
\begin{proof}
Let us define the sets $E_j$ and $F_k$ as in the proof of Theorem~\ref{thm:N1X-complete}, $j,k\in\Nbb$. We have obtained there a subsequence (denoted by $\{u_j\}_{j=1}^\infty$ again) which converges uniformly to some function $\tilde{u}$ on $\Pcal \setminus F_k$ for any $k\in\Nbb$, and thus pointwise on $\Pcal \setminus F$, i.e., q.e.\@ on $\Pcal$. Due to the construction of the function $\tilde{u}$ (which is denoted by $u$ in the aforementioned proof), we see that $\tilde{u}\in \NX$, whence $\|u-\tilde{u}\|_{\NX} = 0$ and Proposition~\ref{pro:N1X-zero-functions} then yields that $u=\tilde{u}$ q.e.

If all functions $u_j$ are continuous, then all sets $E_j$, and consequently $F_k$, are open. Therefore, $G$ can be defined as $F_k$ for a suitably large $k$, and $u_j \to \tilde{u}$ uniformly outside of $G$.
\end{proof}
%
%
%
%

\end{document}